\documentclass[11pt]{amsart}

\usepackage[a4paper,hmargin=3.5cm,vmargin=4cm]{geometry}
\usepackage{amsfonts,amssymb,amscd,amstext,verbatim}

\usepackage{fancyhdr}
\usepackage[utf8]{inputenc}
\usepackage{hyperref}
\usepackage{verbatim}

\usepackage[T1]{fontenc}

\input xy
\xyoption{all}

\usepackage{times}

\usepackage{enumerate}
\usepackage{titlesec}
\usepackage{mathrsfs}

\titleformat{\section}
{\filcenter\bfseries\large} {\thesection{.}}{0.2cm}{}
\titleformat{\subsection}[runin]
{\bfseries} {\thesubsection{.}}{0.15cm}{}[.]
\titleformat{\subsubsection}[runin]
{\em}{\thesubsubsection{.}}{0.15cm}{}[.]

\usepackage[up,bf]{caption}

\newtheorem{theorem}{Theorem}[section]
\newtheorem{proposition}[theorem]{Proposition}

\newtheorem{lemma}[theorem]{Lemma}
\newtheorem{corollary}[theorem]{Corollary}

\theoremstyle{definition}
\newtheorem{definition}[theorem]{Definition}
\newtheorem{remark}[theorem]{Remark}

\newtheorem{problem}[theorem]{Problem}
\newtheorem{example}[theorem]{Example}

\numberwithin{equation}{section}
\numberwithin{figure}{section}

\newcommand\Ccal{\mathcal{C}}

\newcommand\Fcal{\mathcal{F}}

\newcommand\Hcal{\mathcal{H}}

\newcommand\Kcal{\mathcal{K}}

\newcommand\Mcal{\mathcal{M}}
\newcommand\Ncal{\mathcal{N}}

\newcommand\Pcal{\mathcal{P}}

\newcommand\Rcal{\mathcal{R}}

\newcommand\Tcal{\mathcal{T}}

\newcommand\be{\mathbf{e}}
\newcommand\bp{\mathbf{p}}
\newcommand\bq{\mathbf{q}}
\newcommand\bu{\mathbf{u}}
\newcommand\bv{\mathbf{v}}
\newcommand\bx{\mathbf{x}}
\newcommand\by{\mathbf{y}}
\newcommand\bw{\mathbf{w}}
\newcommand\bz{\mathbf{z}}

\newcommand\Cscr{\mathscr{C}}

\newcommand\B{\mathbb{B}}
\newcommand\C{\mathbb{C}}
\newcommand\D{\overline{\mathbb D}}

\renewcommand\D{\mathbb D}

\renewcommand\H{\mathbb{H}}

\newcommand\N{\mathbb{N}}

\newcommand\R{\mathbb{R}}
\newcommand\RP{\mathbb{RP}}

\newcommand\T{\mathbb{T}}
\newcommand\Z{\mathbb{Z}}

\newcommand\igot{\mathfrak{i}}

\renewcommand\igot{\mathfrak{i}}

\newcommand\E{\mathrm{e}}

\renewcommand\imath{\igot}
\newcommand\zero{\mathbf{0}}

\newcommand\di{\partial}

\newcommand\dist{\mathrm{dist}}
\renewcommand\span{\mathrm{span}}

\newcommand\Aut{\mathrm{Aut}}

\newcommand\Id{\mathrm{Id}}

\newcommand\Vol{\mathrm{Vol}}
\newcommand\CF{\mathrm{CF}}
\newcommand\CH{\mathrm{CH}}
\newcommand\HN{\mathrm{HN}}

\def\dist{\mathrm{dist}}
\def\span{\mathrm{span}}

\newcommand\CK{\mathcal{CK}}

\newcommand\nullq{{\mathbf A}}

\numberwithin{equation}{section}

\begin{document}

\fancyhead[LO]{Schwarz--Pick lemma for harmonic maps which are conformal at a point}
\fancyhead[RE]{F.\ Forstneri{\v c} and D.\ Kalaj}
\fancyhead[RO,LE]{\thepage}

\thispagestyle{empty}

\vspace*{1cm}
\begin{center}
{\bf\LARGE Schwarz--Pick lemma for harmonic maps \\ which are conformal at a point}

\vspace*{0.5cm}

{\large\bf  Franc Forstneri{\v c} and David Kalaj}
\end{center}

\vspace*{1cm}

\begin{quote}
{\small
\noindent {\bf Abstract}\hspace*{0.1cm}
In this paper we obtain a sharp estimate on the norm of the differential of a harmonic map 
from the unit disc $\D$ in $\C$ into the unit ball $\B^n$ of $\R^n$, $n\ge 2$, at any point where 
the map is conformal. For $n=2$ this generalizes the classical Schwarz--Pick lemma,
and for $n\ge 3$ it gives the optimal Schwarz--Pick lemma for conformal minimal 
discs $\D\to \B^n$. This implies that conformal harmonic maps 
$M \to \B^n$ from any hyperbolic conformal surface are distance-decreasing in the Poincar\'e metric 
on $M$ and the Cayley--Klein metric on the ball $\B^n$, and 
the extremal maps are  the conformal embeddings of the disc 
$\D$ onto affine discs in $\B^n$. Motivated by these results, we introduce an 
intrinsic pseudometric on any Riemannian manifold of dimension at least three
using conformal minimal discs, and we lay foundations of  the corresponding 
hyperbolicity theory.  

\vspace*{0.2cm}

\noindent{\bf Keywords}\hspace*{0.1cm}  harmonic map, conformal minimal surface, 
Schwarz--Pick lemma, Cayley--Klein metric 

\vspace*{0.1cm}

\noindent{\bf MSC (2020):}\hspace*{0.1cm}}
Primary: 53A10; Secondary: 30C80, 31A05, 32Q45

\vspace*{0.1cm}
\noindent{\bf Date: 
1 December 2021}

%
%
%
%
%
%
\end{quote}

%
%
%
%
\section{Introduction}\label{sec:intro} 
In this paper we establish precise estimates of derivatives and the rate of growth of 
conformal harmonic maps from hyperbolic conformal surfaces into the unit ball $\B^n$ 
of $\R^n$ for any $n\ge 3$; see Theorem \ref{th:metricdecreasing}. 
Such maps parameterize minimal surfaces, objects of high interest in geometry. 
To motivate the discussion, we begin with the following special case 
of one of our main results, Theorem 2.1. This generalizes the classical 
Schwarz--Pick lemma, due to H.\ A.\ Schwarz \cite[Bd. II, p.\ 108]{Schwarz1869} (1869),  
H.\ Poincar\'e \cite{Poincare1884} (1884), C.\ Carath\'eodory \cite{Caratheodory1912} (1912), 
and G.\ A.\ Pick \cite{Pick1915} (1915), to a substantially larger class of maps. 

%
%
\begin{theorem}\label{th:SP}
Let $\D=\{z\in\C:|z|<1\}$ denote the unit disc.
If $f:\D\to \D$ is a harmonic map which is conformal at a point $z\in \D$, then at this point we have that 
\vspace{-1mm}
\begin{equation}\label{eq:SP}
	\|df_z\| \le \frac{1-|f(z)|^2}{1-|z|^2},
\end{equation}
with equality if and only if $f$ is a conformal diffeomorphism of the disc $\D$.
\end{theorem}

The classical Schwarz--Pick lemma gives the same conclusion under the much stronger
hypothesis that the map $f$ is holomorphic or antiholomorphic, 
which means that it is conformal at every noncritical point;
see e.g.\ \cite{Dineen1989,Kobayashi2005,Royden1988}. 
This fundamental rigidity result in complex analysis leads to the notion of Kobayashi hyperbolic manifolds \cite{Kobayashi1967,Kobayashi1976,Kobayashi2005}
and provides a connection to complex differential geometry via the Ahlfors lemma 
(see \cite{Ahlfors1938}, \cite[Theorem 2.1]{Kobayashi2005}, \cite{Royden1988}) 
and its generalizations by S.-T.\ Yau \cite{Yau1978} and others.

The conditions in Theorem \ref{th:SP} are invariant under precompositions by 
holomorphic automorphisms of $\D$, so the proof reduces to the case 
$z=0$. On the other hand, postcompositions of harmonic maps into $\D$
by holomorphic automorphism of $\D$ need not be harmonic, so we cannot exchange $f(0)$ and $0$.
Hence, the standard proof of the classical Schwarz--Pick lemma breaks down.
The estimate \eqref{eq:SP} fails for some nonconformal 
harmonic diffeomorphisms of $\D$ (see Example \ref{ex:failure}), as well as for harmonic 
maps $\D\to D$ to more general domains which are conformal at a point (see Example \ref{ex:square} and Problem \ref{prob:disc}). 

The main results of the paper are precise estimates of the differential, and of the rate of growth 
of conformal harmonic maps $M\to \B^n$ from an open conformal surface, $M$, 
to the unit ball $\B^n$ of $\R^n$ for any $n\ge 3$.
It is classical that such maps parameterize minimal surfaces.
Indeed, a smooth conformal map $f:M\to \R^n$ 
from an open conformal surface $M$ into $\R^n$ with the Euclidean metric 
parameterizes a minimal surface in $\R^n$ if and only if $f$ is a harmonic map 
(see Osserman \cite{Osserman1986}, Duren \cite{Duren2004}, 
Alarc\'on et al.\ \cite[Chapter 2]{AlarconForstnericLopez2021}, among other sources). 
Note that an oriented conformal surface is a Riemann surface. 

The focal point of the paper is Theorem \ref{th:main}, which 
gives a precise upper bound on the norm $\|df_z\|$ 
of the differential $df_z$ of a harmonic map $f:\D\to\B^n$ at any point
$z\in \D$ where the map is conformal. The estimate is similar to the one in Theorem \ref{th:SP},
except that for $n\ge 3$ it also involves the angle $\theta$ between the position vector $f(z)\in\B^n$ and the 
2-plane $df_z(\R^2)\subset \R^n$. A related result (see Theorem \ref{th:main2}) shows that the worst 
case estimate, which occurs for $\theta=\pi/2$ (i.e., 
when the vector $f(z)$ is orthogonal to the plane $df_z(\R^2)$),  
holds for all harmonic maps $f:\D\to\B^n$ provided $\|df_z\|$ is replaced by 
$\sqrt2^{-1}|\nabla f(z)|$; these quantities coincide if $f$ is conformal at $z$.

We then give a differential geometric formulation and an extension of Theorem \ref{th:main}.
Let $\Ccal\Kcal$ denote the Cayley-Klein metric on the ball $\B^n$ $(n\ge 2)$, also called
the Beltrami--Klein metric; see \eqref{eq:CK} and the footnote
on p.\ \pageref{p:CK}. This metric is one of the classical models of hyperbolic geometry.
It coincides with the restriction of the Kobayashi metric on the 
complex ball $\B^n_\C\subset \C^n$ \eqref{eq:BC} (which is the same as 
$1/\sqrt{n+1}$ times the Bergman metric on $\B^n_\C$) to points of the real ball $\B^n$ and real tangent vectors.
Theorem \ref{th:main} implies that any conformal harmonic map $f:M\to \B^n,\ n\ge 3$, from a 
hyperbolic conformal surface is metric- and distance-decreasing in the Poincar\'e metric on $M$ 
and the Cayley--Klein metric on $\B^n$; see Theorem \ref{th:metricdecreasing}.   
Furthermore, if the differential $df_p$ has the operator norm equal to $1$ at some point $p\in M$
in this pair of metrics, or if $f$ preserves the distance between a pair of distinct points in $M$, 
then $M$ is necessarily the disc $\D$ and $f$ is a conformal diffeomorphism of 
$\D$ onto a proper affine disc in $\B^n$.  
In particular, a conformal harmonic disc $f:\D\to \B^n$ with $f(0)=0$ satisfies 
$|f(z)|\le |z|$ for all $z\in \D$ (see Corollary \ref{cor:disc0}). 

In Section \ref{sec:results} we give precise statements of the mentioned results. 
Theorem \ref{th:main} is proved in Section \ref{sec:prep}. We introduce a new idea into the subject, 
connecting it to Lempert's seminal work \cite{Lempert1981} from 1981
on complex geodesics of the Kobayashi metric on bounded convex domains in $\C^n$.
Theorem \ref{th:main2} is proven in Section \ref{sec:proofmain2}.  In Section \ref{sec:Beltrami}
we apply Theorem \ref{th:SP} to estimate the gradient of a quasiconformal harmonic self-map of the disc 
in terms of its second Beltrami coefficient at the reference point; see Theorem \ref{th:Beltrami}. 

Motivated by these result, we introduce in Section \ref{sec:hyperbolicity} an intrinsic
pseudometric on any domain in $\R^n,\ n\ge 3$ (and more generally on any
Riemannian manifold of dimension at least three) in terms of conformal
minimal discs, in analogy to Kobayashi's definition of his pseudometric on complex manifolds
in terms of holomorphic discs. This provides the basis for a new hyperbolicity theory
of such domains, and of Riemannian manifolds.

%
%
\section{The main results}\label{sec:results}
Given a differentiable map $f:\D\to \R^n$, we denote by $f_x$ and $f_y$
its partial derivatives with respect to $x$ and $y$, where $z=x+\imath y\in \D$. The gradient
$\nabla f=(f_x,f_y)$ is an $n\times 2$ matrix representing the differential $df$. 
The map $f$ is said to be {\em conformal} at $z\in \D$ if 
\begin{equation}\label{eq:conformal}
	|f_x(z)|=|f_y(z)| \quad \text{and}\quad f_x(z)\,\cdotp f_y(z)=0.
\end{equation}
Here, the dot stands for the Euclidean inner product on $\R^n$, and $|\bx|$ is the Euclidean norm 
of $\bx\in\R^n$. If $f$ is an immersion at $z$ then \eqref{eq:conformal} holds if and only if 
$df_z$ preserves angles. It follows from \eqref{eq:conformal} that $f$ has rank zero at any branch point.
We denote by $|\nabla f|$ the Euclidean norm of the gradient: 
\[
	|\nabla f(z)|^2=|f_x(z)|^2+|f_y(z)|^2, \quad\ z\in\D.
\]
If $f$ is conformal at $z$ then clearly 
$
	 \|df_z\| = {\sqrt 2}^{-1} |\nabla f(z)|= |f_x(z)| =  |f_y(z)|.
$
The map $f=(f_1,\ldots,f_n):\D\to\R^n$ is harmonic if and only if every component
$f_k$ is a harmonic function on $\D$, meaning that the Laplacian 
$\Delta f_k=\frac{\di^2 f_k}{\di x^2} + \frac{\di^2 f_k}{\di y^2}$
vanishes identically. 

We denote by $\B^n$ the unit ball of $\R^n$:
\begin{equation}\label{eq:ball}
	\B^n=\Big\{\bx=(x_1,\ldots,x_n)\in\R^n: |\bx|^2=\sum_{k=1}^n x_k^2<1\Big\}. 
\end{equation}
Our first main result is the following; it is proved in Section \ref{sec:prep}.

%
%
\begin{theorem} \label{th:main}
Let $f:\D\to\B^n$ for $n\ge 2$ be a harmonic map.
If $f$ is conformal at a point $z\in\D$ and $\theta \in [0,\pi/2]$ denotes the angle between 
the vector $f(z)$ and the plane $\Lambda=df_z(\R^2)\subset\R^n$, then 
\begin{equation}\label{eq:S}
	  \|df_z\| = \frac{1}{\sqrt{2}}\, |\nabla f(z)|  \ \le\ 
	  \frac{1-|f(z)|^2}{1-|z|^2}  \frac{1}{\sqrt{1-|f(z)|^2 \sin^2\theta}}, 
\end{equation}
with equality if and only if $f$ is a conformal diffeomorphism of $\D$ 
onto the affine disc $\Sigma=(f(z)+\Lambda) \cap\B^n$. (When $f(z)=0$
or $df_z=0$, the angle $\theta$ does not matter.) 
\end{theorem}

Note that the number $R=\sqrt{1-|f(z)|^2 \sin^2\theta}$ is the radius of the affine disc $\Sigma$.
In dimension $n=2$ we have $\theta=0$, so Theorem \ref{th:SP} 
is a special case of Theorem \ref{th:main}.
Without assuming that $f$ is conformal at $z$ or that $f(z)=0$, the inequality
\eqref{eq:S} fails for some harmonic diffeomorphisms of the disc as shown by Example \ref{ex:failure}.

For a fixed value of $|f(z)| \in [0,1)$, the maximum of the right hand side of \eqref{eq:S} 
over angles $\theta\in [0,\pi/2]$ equals $\frac{\sqrt{1-|f(z)|^2}}{1-|z|^2}$ and is reached 
precisely at $\theta=\pi/2$, i.e, when the vector $f(z)$ is orthogonal to $\Lambda=df_z(\R^2)$,
unless $f(z)=0$ when it is independent of $\theta$. 
It turns out that this weaker estimate holds for all harmonic maps $\D\to \B^n$
without any conformality assumption. The following result is proved in Section \ref{sec:proofmain2}.

%
%
\begin{theorem}\label{th:main2}
For every harmonic map $f:\D\to \B^n$  $(n\ge 2)$ we have that
\begin{equation}\label{eq:S2}
		\frac{1}{\sqrt{2}} |\nabla f(z)|  \ \le\  \frac{\sqrt{1-|f(z)|^2}}{1-|z|^2}, \quad\ z\in\D.
\end{equation}
Equality holds for some $z\in \D$ if $f(z)$ is orthogonal to the $2$-plane $\Lambda=df_{z}(\R^2)$ 
and $f$ is a conformal diffeomorphism onto the affine disc $\left(f(z)+\Lambda\right)\cap \B^n$.
In particular, if $f(z)=0$ then $|\nabla f(z)| \le \frac{\sqrt{2}}{1-|z|^2}$,
with equality if and only if $f$ is a conformal diffeomorphism onto the linear disc $\Lambda \cap \B^n$.
\end{theorem}
 
The estimate \eqref{eq:S2} only uses the hypothesis that the $L^1$-norm 
of $|f|^2=\sum_{k=1}^n f_k^2$ on the circles $\{|z|=r\}$ for $0<r<1$ is bounded by $1$. 
This clearly holds for maps into the ball; however, we do not know whether there are harmonic maps 
reaching (near) equality in \eqref{eq:S2} whose images are actually contained 
in the ball, except in the special cases indicated in the theorem. 
In this connection, see the discussion following Theorem \ref{th:Beltrami}. 

The precise upper bound on the size of the gradient $\|df_0\|$ of a nonconformal harmonic map
$f:\D\to\B^n$ with a given centre $f(0)=\bx\in \B^n\setminus\{0\}$ for $n\ge 2$
in terms of the distortion of $f$ at $0$ is unknown;
see the papers by Kovalev and Yang \cite{KovalevYang2020} and Brevig et al.\ \cite{BrevigOrtegaSeip2021} for $n=2$. 
On the other hand, for $n=1$ the {\em harmonic Schwarz lemma} 
(see Axler et al.\ \cite[Theorem 6.26]{AxlerBourdonRamey2001})
says that any harmonic function $f:\B^m \to (-1,+1)$ for $m\ge 2$ satisfies the sharp estimate
$
	|\nabla f(0)| \le \frac{2\Vol(\B^{m-1})}{\Vol(\B^m)}.
$ 
For $m=2$ the inequality reads
$
	|\nabla f(0)| \le \frac{4}{\pi},
$
and a simple proof in this case was given by Kalaj and Vuorinen \cite[Theorem 1.8]{KalajVuorinen2012}.

%
%
Let us mention a consequence of Theorem \ref{th:main} 
related to the Schwarz lemma for holomorphic discs 
in the ball of the complex Euclidean space:
\begin{equation}\label{eq:BC}
	\B^n_\C=\Big\{\bz=(z_1,\ldots,z_n)\in\C^n: |\bz|^2=\sum_{k=1}^n |z_k|^2<1\Big\}
\end{equation}
(see Rudin \cite[Sect.\ 8.1]{Rudin2008}). The following corollary to Theorem \ref{th:main} 
shows that the extremal holomorphic discs in $\B^n_\C$ are precisely those extremal orientation 
preserving conformal harmonic discs $\D\to\B^n_\C$ which parameterize affine complex discs.

%
%

\begin{corollary} 
\label{cor:Schwarzholo}
Let $f:\D\to\B^n_\C$ be a harmonic map which is conformal at a point $z\in \D$. 
If $\Lambda = df_{z}(\R^2)$ is a complex line in $\C^n$, then equality holds in \eqref{eq:S}
for this $z$ if and only if $f$ is a biholomorphic or anti-biholomorphic map onto the affine complex disc
$(f(z)+\Lambda)\cap \B^n_\C$.
\end{corollary}

%
%
\noindent{\bf The Cayley--Klein metric.} 
A differential geometric interpretation of the classical Schwarz--Pick lemma is that
holomorphic maps $\D\to\D$ are distance-decreasing in the Poincar\'e metric on $\D$,
and isometries coincide with holomorphic and antiholomorphic automorphisms of $\D$ 
(see Kobayashi \cite{Kobayashi2005}).
The analogous conclusion holds for holomorphic maps $\D\to \B^n_\C$ 
with the Kobayashi metric on $\B^n_\C$ \eqref{eq:BC}, where orientation-preserving isometric 
embeddings are precisely holomorphic embeddings onto affine complex discs in $\B^n_\C$. 

In the same spirit, we shall now interpret Theorem \ref{th:main} as the distance-decreasing property 
of conformal harmonic maps $\D\to\B^n$ with respect to the \label{p:CK}
{\em Cayley--Klein metric}\footnote{The Beltrami--Calvin--Klein model 
of hyperbolic geometry was introduced by Arthur Cayley \cite{Cayley1859} (1859) 
and Eugenio Beltrami \cite{Beltrami1868} (1868), and it was developed by Felix Klein 
\cite{Klein1871,Klein1873} (1871, 1873). The underlying space is the $n$-dimensional unit ball,
geodesics are straight line segments with ideal endpoints on the boundary sphere, and the distance 
between points on a geodesic is given by a cross ratio. 
This is a special case of the Hilbert metric on convex domains in $\R^n$ and $\RP^n$,  
introduced by David Hilbert in 1885 \cite{Hilbert1885}. These are examples of 
projectively invariant metrics discussed by many authors; see the surveys 
by S.\ Kobayashi \cite{Kobayashi1977,Kobayashi1984}, W.\ M.\ Goldman \cite{Goldman2019}, 
and J.\ G.\ Ratcliffe \cite{Ratcliffe1994}.}
on $\B^n$:
\begin{equation}\label{eq:CK}
	\CK(\bx,\bv) = \frac{\sqrt{1-|\bx|^2 \sin^2\phi}}{1-|\bx|^2} \, |\bv|, \quad\  \bx\in\B^n,\ \bv\in \R^n,
\end{equation}
where $\phi \in [0,\pi/2]$ is the angle between the vector $\bx$ and the line $\R\bv$. Equivalently,
\begin{equation}\label{eq:CKsquared}
	\CK(\bx,\bv)^2 = \frac{(1-|\bx|^2)|\bv|^2 + |\bx\,\cdotp \bv|^2}{(1-|\bx|^2)^2}
	= \frac{|\bv|^2}{1-|\bx|^2} + \frac{|\bx\,\cdotp \bv|^2}{(1-|\bx|^2)^2}.
\end{equation}
Let $G_2(\R^n)$ denote the Grassmann manifold of $2$-planes in $\R^n$.
We define a Finsler pseudometric $\Mcal:\B^n\times  G_2(\R^n)\to \R_+$ by
\begin{equation}\label{eq:Mcal}
	\Mcal(\bx,\Lambda) \ =\ \frac{\sqrt{1-|\bx|^2 \sin^2\theta}} {1-|\bx|^2},\quad \ 
	\bx\in \B^n, \ \Lambda\in  G_2(\R^n),
\end{equation}
where $\theta\in[0,\pi/2]$ is the angle between $\bx$ and $\Lambda$. 
At $\bx=\zero$ we have $\Mcal(\zero,\Lambda)=1$ for all $\Lambda\in  G_2(\R^n)$.
Assume now that $\bx\ne\zero$. 
Let $\bv\in\R^n\setminus \{\zero\}$ be a vector having angle $\phi\in[0,\pi/2]$ with the line $\R\bx$.
The angle $\theta$ between $\bx$ and any 2-plane $\Lambda$ containing $\bv$ 
satisfies $0\le \theta\le \phi$, and the maximum of $\theta$ over all such 
$\Lambda$ equals $\phi$. Hence, \eqref{eq:CK} gives
\begin{eqnarray}\label{eq:min}
	\CK(\bx,\bv)/|\bv| \ &=& \ \min\bigl\{ \Mcal(\bx,\Lambda): \Lambda\in  G_2(\R^n),\ \bv\in\Lambda\bigr\},\\
	\label{eq:max}
	\Mcal(\bx,\Lambda) \ &=& \ \max\bigl\{\CK(\bx,\bv)/|\bv| : \bv\in\Lambda\bigr\}.
\end{eqnarray}

The inequality \eqref{eq:S} in Theorem \ref{th:main} is obviously equivalent to  
\begin{equation}\label{eq:Sbis}
	 \Mcal(f(z),df_z(\R^2)) \,  |df_z(\xi)|
	  =    \frac{\sqrt{1-|f(z)|^2 \sin^2\theta}}{1-|f(z)|^2} \,  |df_z(\xi)|   
	  \ \le\ \frac{|\xi|} {1-|z|^2}, 
\end{equation}
where $\theta\in[0,\pi/2]$ is the angle between $f(z)$ and the 2-plane $\Lambda=df_z(\R^2)$.
By \eqref{eq:min} the left hand side of \eqref{eq:Sbis} is bigger than or equal to $\CK(f(z),df_z(\xi))$.
Equality holds if and only if the angle $\phi$ between the line $f(z)\R$ and the vector 
$df_z(\xi)\in \Lambda$ equals $\theta$; clearly this holds if and only if $df_z(\xi)$ is tangent 
to the diameter of the affine disc $\Sigma=(f(z)+\Lambda)\cap \B^n$ through the point $f(z)$.
This, and the addition concerning equality in \eqref{eq:S}, give the following corollary
to Theorem \ref{th:main}. Note that $\Pcal_\D(z,\xi)=\frac{|\xi|}{1-|z|^2}$ is the 
Poincar\'e metric on the disc.

%
%
\begin{corollary}\label{cor:metricdecreasing}
If $f:\D\to\B^n$ is a conformal harmonic map then for every $z\in\D$ and $\xi\in\R^2$ we have that
\begin{equation}\label{eq:metricdecreasing}
	\CK\big(f(z),df_z(\xi)\big) \le \frac{|\xi|}{1-|z|^2} = \Pcal_\D(z,\xi). 
\end{equation}
Equality holds for some $z\in \D$ and $\xi\in \R^2\setminus \{0\}$ if and only if 
$f$ is a conformal diffeomorphism onto the affine disc $\Sigma=(f(z)+df_z(\R^2))\cap \B^n$ 
and the vector $df_z(\xi)$ is tangent to the diameter of $\Sigma$ through the point $f(z)$.
\end{corollary}

This shows in particular that every linear conformal embedding
$f:\D\to\Sigma$ onto a proper affine disc in $\B^n$ is geodesic on each diameter
$(-1,+1)\ni r\mapsto f(re^{\imath t})\in \Sigma$ for every fixed $t\in\R$. 
However, distances between points of different rays strictly decrease from the Poincar\'e metric on $\D$
to the Cayley--Klein metric on the disc $\Sigma\subset \B^n$.

%
%
\begin{remark}\label{rem:miminalmetric}
The Cayley--Klein metric \eqref{eq:CKsquared} is the 
restriction of the Kobayashi metric on the unit ball $\B^n_\C \subset \C^n$ to points 
$\bx\in\B^n=\B^n_\C\cap\R^n$ of the real ball and tangent vectors in $T_\bx\R^n\cong\R^n$. 
A direct geometric argument was given by Lempert in \cite[proof of Theorem 3.1]{Lempert1993}.
The Cayley--Klein metric also equals $1/\sqrt{n+1}$ times the Bergman metric on $\B^n_\C$ 
restricted to $\B^n$ and real tangent vectors; see Krantz \cite[Proposition 1.4.22]{Krantz1992}. 
(On the ball of $\C^n$, most holomorphically invariant metrics coincide up to scalar factors.)
The Cayley--Klein metric equals the Poincar\'e metric $\frac{|\bv|}{1-|\bx|^2}$ on $\B^n$ 
on vectors $\bv$ parallel to the base point $\bx\in\B^n$, but is strictly smaller 
on other vectors. While the Poincar\'e metric on $\B^n$ is conformally equivalent to the 
Euclidean metric, the Cayley--Klein metric is not.   
\end{remark}

%
%
We now extend Corollary \ref{cor:metricdecreasing} to more general minimal surfaces.
A {\em conformal surface} is a topological surface, $M$, together with a {\em conformal atlas},
i.e., one whose transition maps between charts are conformal diffeomorphisms between plane domains. 
Every surface admits a conformal structure. Indeed, every topological 
surface admits a smoothing, and a conformal structure on a smooth surface is determined by the choice 
of a Riemannian metric in view of the existence of local isothermal coordinates 
(see \cite{Osserman1986} or \cite[Theorem 1.8.6]{AlarconForstnericLopez2021}).  
Oriented conformal surfaces are Riemann surfaces.
There is a well-defined notion of a harmonic function
on a conformal surface. Indeed, a Riemannian metric $g$ 
defines the metric Laplacian $\Delta_g$, hence $g$-harmonic functions 
satisfying $\Delta_g h=0$. The Laplacians associated to any two Riemannian metrics in the same 
conformal class on a surface differ by a positive multiplicative function  
(see \cite[Corollary 1.8.2]{AlarconForstnericLopez2021}), and hence the notion of a harmonic function
is independent of the choice of metric in a given conformal class. 

A conformal surface $M$ is said to be {\em hyperbolic} 
if its universal conformal covering space is the disc $\D$.
Let $h:\D\to M$ be a universal conformal covering map. Since conformal automorphisms of $\D$ 
are isometries of the Poincar\'e metric $\Pcal_\D= \frac{|dz|}{1-|z|^2}$, 
there is a unique Riemannian metric $\Pcal_M$ on $M$ (a K\"ahler metric if $M$ is a Riemann surface) 
such that $h$ is a local isometry. This Poincar\'e metric $\Pcal_M$ is a
complete metric of constant Gaussian curvature $-4$ (see \cite[p.\ 48, Example 2]{Kobayashi2005}),
which agrees with the Kobayashi metric if $M$ is a Riemann surface. 
This leads to the following generalization of Corollary \ref{cor:metricdecreasing}.

%
%
\begin{theorem}[Metric- and distance-decreasing property of conformal harmonic maps] 
\label{th:metricdecreasing}
Let $M$ be a connected hyperbolic conformal surface endowed with the Poincar\'e
metric $\Pcal_M$. Every conformal harmonic map $f:M\to \B^n$ $(n\ge 3)$ satisfies 
\begin{equation}\label{eq:dd2}
	\CK\big(f(p),df_p(\xi)\big) \ \le\ \Pcal_M(p,\xi),\quad\  p\in M,\ \xi\in T_pM.
\end{equation}
Hence, $f$ is distance-decreasing in this pair of metrics.
If equality holds in \eqref{eq:dd2} for some point $p\in M$ and vector $0\ne \xi\in T_pM$,
or if $f$ preserves the distance on a pair of distinct points in $M$, 
then $M=\D$ and $f$ is a conformal diffeomorphism onto an affine disc in $\B^n$. 
\end{theorem}

Note that Theorem \ref{th:metricdecreasing} is nontrivial only if the surface $M$ is 
{\em of hyperbolic type}, i.e.\ it admits a nonconstant bounded harmonic function 
(see Farkas and Kra \cite[p.\ 179]{FarkasKra1992} and Grigor'yan \cite{Grigoryan1999BAMS}).
Every surface of hyperbolic type is also hyperbolic.

\begin{proof}
Assume first that $M$ is orientable, hence a Riemann surface.
Choose a holomorphic covering map $h:\D\to M$ and a point $z\in \D$ with $h(z)=p$.
The conformal harmonic map $\tilde f = f\circ h:\D \to \B^n$ then satisfies $\tilde f(z)=f(p)$
and $d\tilde f_z=df_p\circ dh_z$. Let $\eta\in\R^2$ be such that $dh_z(\eta)=\xi$.
Then, $\Pcal_M(p,\xi)= \Pcal_\D(z,\eta)$ by the definition of the metric $\Pcal_M$,
and $d\tilde f_z(\eta) = df_p(\xi)$. From \eqref{eq:metricdecreasing} it follows that 
\[
	\CK\big(f(p),df_p(\xi)\big) = 
	\CK\big(\tilde f(z),d\tilde f_z(\eta)\big) \le \frac{|d\tilde f_z(\eta)|}{1-|\tilde f(z)|^2}
	= \frac{|df_p(\xi)|}{1-|f(p)|^2},
\]
which gives \eqref{eq:dd2}. If $\xi\ne 0$ and equality holds, then by Corollary \ref{cor:metricdecreasing} 
the map $\tilde f=f\circ h:\D\to \B^n$ is a conformal diffeomorphism onto an affine disc in $\B^n$, 
and hence $h:\D\to M$ is a biholomorphism. For a nonorientable hyperbolic conformal surface 
$M$ we obtain the same conclusion by passing to its orientable 2-sheeted conformal cover.
The statement concerning distances is an immediate consequence. 
Note that if the distances agree for a pair of distinct points 
in $M$ and their images in $\B^n$, then the differential $df_p$ has operator norm $1$ at some 
point $p\in M$ in the given pair of metrics.
\end{proof}

On the disc with the Poincar\'e metric $\Pcal_\D=\frac{|dz|}{1-|z|^2}$, the Poincar\'e distance equals
\begin{equation}\label{eq:ballP}
	\dist_\Pcal(z,w)=\frac12 \log\left( \frac{|1- z\overline w|+ |z-w|}{|1- z\overline w| - |z-w|}\right),
	\quad z,w\in\D.
\end{equation}
The Cayley--Klein distance function on the ball $\B^n$   
coincides up to a scalar factor $\sqrt{n+1}$ with the restriction to $\B^n$ of the
Bergman distance function on the complex ball $\B^n_\C$ or, equivalently, with the restriction to 
$\B^n$ of the Kobayashi distance function on $\B^n_\C$. The following explicit formula for the 
Kobayashi distance between a pair of points $\bz,\bw\in\B^n_\C$
can be found in \cite[p.\ 437]{Krantz1992} 
(here, $\bz\,\cdotp\overline \bw=\sum_{k=1}^n z_k\overline w_k$).
\begin{equation}\label{eq:ballM}
	\dist(\bz,\bw)=\frac12 
	\log\left( 
	\frac{|1- \bz\,\cdotp\overline \bw| + \sqrt{|\bz-\bw|^2 +|\bz\,\cdotp\overline \bw|^2 -|\bz|^2|\bw|^2}}
	{|1- \bz\,\cdotp\overline \bw| - \sqrt{|\bz-\bw|^2 +|\bz\,\cdotp\overline \bw|^2 -|\bz|^2|\bw|^2}} \right).
\end{equation}
As said before, the same formula applied to points in $\B^n$Gives the Cayley--Klein distance.
Taking $w=0$ and $\bw=0$ in the above formulas, we obtain 
\[
	\dist_{\Pcal}(z,0)=\frac12  \log\left( \frac{1+ |z|}{1-|z|}\right) \ \ (z\in\D),\quad
	\dist(\bz,\zero)      =\frac12 \log\left( \frac{1+ |\bz|}{1-|\bz|}\right) \ \ (\bz\in\B^n).
\]
Together with Theorem \ref{th:metricdecreasing} this gives implies the following corollary.

\begin{corollary}\label{cor:disc0}
If $f:\D\to\B^n,\ n\ge 3,$ is a conformal harmonic map with $f(0)=\zero$, then $|f(z)|\le |z|$ 
holds for all $z\in \D$.  Equality at one point $z\in \D\setminus\{0\}$ implies that $f$ is a conformal 
parameterization of a linear disc obtained by intersecting $\B^n$ with a plane through the origin, 
and hence equality holds at all points.
\end{corollary}

%
%
\section{Proof of Theorem \ref{th:main}}\label{sec:prep}
It suffices to prove Theorem \ref{th:main} for $z=0$.
Indeed, with $f$ and $z$ as in the theorem, let $\phi_z\in\Aut(\D)$ be 
such that $\phi_z(0)=z$. The harmonic map $g=f\circ\phi_z:\D\to \B^n$ is then conformal at the origin.
Since $|\phi_z'(0)|=1-|z|^2$, \eqref{eq:S} follows from the same estimate for $g$
at $z=0$. On the image side, the hypotheses and the statement of the theorem
are invariant under postcomposition of maps $\D\to\B^n$ by
elements of the orthogonal group $O_n$.

We begin by an explicit description of conformal parameterizations of proper affine discs in $\B^n$.
Fix a point $\bq\in \B^n$ and a linear $2$-plane $0\in \Lambda\subset \R^n$, and
consider the affine disc $\Sigma = (\bq+\Lambda)\cap \B^n$. Let us identify
conformal parameterizations $\D\to\Sigma$ sending $0$ to $\bq$.
Let $\bp\in\Sigma$ be the closest point to the origin. If $n=2$ then $\bp=0$ and $\Sigma=\D$.
Suppose now that $n\ge 3$. Up to an orthogonal rotation, we may assume that
\begin{equation}\label{eq:Sigmaf2}
	\bp=(0,0,p,0\ldots,0)\quad\text{and}\quad \Sigma=\left\{(x,y,p,0,\ldots,0): x^2+y^2 <1-p^2 \right\}.
\end{equation}
Let $\bq=(b_1,b_2,p,0,\ldots,0)\in\Sigma$, and let $\theta$ denote the angle between
$\bq$ and $\Sigma$. Set
\begin{equation}\label{eq:constants}
	c=\sqrt{1-p^2}=\sqrt{1-|\bq|^2\sin^2\theta},\quad\
	a=\frac{b_1+\imath b_2}{c}\in\D,\quad\  |a|= \frac{|\bq|\cos\theta}{c}.
\end{equation}
We orient $\Sigma$ by the tangent vectors $\di_x,\di_y$
in the parameterization \eqref{eq:Sigmaf2}. Every orientation preserving
conformal parameterization $f:\D\to \Sigma$ with $f(0)=\bq$  is then of the form
\begin{equation}\label{eq:extremal}
	f(z)=\left(c\, \Re \frac{e^{\imath t}z+a}{1+\bar a e^{\imath t}z},
	c\, \Im \frac{e^{\imath t}z+a}{1+\bar a e^{\imath t}z},p,0,\ldots,0\right),\quad z\in \D
\end{equation}
for some $t\in\R$. (Here, $\Re$ and $\Im$ stand for the real and imaginary part of a complex number.
If $n=2$ then $p=0$, $c=1$, and the same holds if we drop all coordinates except the first two.
Orientation reversing conformal parameterizations are obtained by replacing $z=x+\imath y$ with 
$\bar z=x-\imath y$. By a rotation in the $(x,y)$-plane, we may further assume that $b_2=0$ and
$f(0)=(b_1,0,p,0,\ldots,0)$; in this case $a\in [0,1)$.
By also allowing rotations on the disc $\D$, we can take $t=0$ in \eqref{eq:extremal}.) 
Using the complex coordinate $x+\imath y$ in the plane $df_0(\R^2)=\R^2\times\{0\}^{n-2}$,
the map \eqref{eq:extremal} can be written in the form 
\[	
	f(z)= \left(c \frac{e^{\imath t}z+a}{1+\bar a e^{\imath t}z},p,0,\ldots,0\right) = (h(z),p,0,\ldots,0).
\]
From \eqref{eq:constants} it follows that
\begin{eqnarray*}
	|h'(0)| &=& c\, (1-|a|^2) = \frac{c^2-c^2 |a|^2}{c}
	         = \frac{1-|\bq|^2\sin^2\theta - |\bq|^2\cos^2\theta}{c} \\
	         &=& \frac{1-|\bq|^2}{\sqrt{1- |\bq|^2\sin^2\theta}}  
	         = \frac{1-|f(0)|^2}{\sqrt{1- |f(0)|^2\sin^2\theta}}.
\end{eqnarray*}
Since $\|df_0\|=|h'(0)|$, this gives equality in \eqref{eq:S} at $z=0$.

Theorem \ref{th:main} now follows immediately from the following lemma.

%
%
\begin{lemma}\label{lem:main}
Let $f:\D\to\B^n$ $(n\ge 2)$ be the disc \eqref{eq:extremal}.  
If $g:\D\to\B^n$ is a harmonic disc such that $g(0)=f(0)$, 
$g$ is conformal at $0$, and $dg_0(\R^2)=df_0(\R^2)$, then $\|dg_0\|\le \|df_0\|$, 
with equality if and only if $g(z)=f(e^{\imath s}z)$ or $g(z)=f(e^{\imath s}\bar z)$ 
for some $s\in\R$ and all $z\in\D$.
\end{lemma}

The proof of Lemma \ref{lem:main} uses ideas from Lempert's seminal
paper \cite{Lempert1981} concerning complex geodesics of the Kobayashi metric
in convex domains in $\C^n$; see Remark \ref{rem:stationary}.

\begin{proof}
Let $p$, $c$ and $a$ be as in \eqref{eq:constants} related to the 
map $f$ in \eqref{eq:extremal}, where $\bq=f(0)$.
Precomposing $f$ by a rotation in $\C$, we may assume that $t=0$ in \eqref{eq:extremal}.
For simplicity of notation we assume that $n=3$; the proof for $n\ne 3$ is exactly the same.
If $n=2$, we delete the remaining components and take $c=1$. 

Consider the holomorphic disc $F:\D\to \Omega=\B^3\times \imath \R^3$Given by
\begin{equation}\label{eq:Fdisc}
	F(z)=\left(c\, \frac{z+a}{1+\bar a z}, -c\,\imath \frac{z+a}{1+\bar a z},p\right),\quad z\in \D.
\end{equation}
Then, $f=\Re F$. Suppose that $g:\D\to\B^3$ is as in the lemma.
Up to replacing  $g(z)$ by $g(e^{\imath s}z)$ or $g(e^{\imath s}\bar z)$ for a suitable $s\in\R$,
we may assume that
\begin{equation}\label{eq:dfdg}
	dg_0=r df_0 \ \ \text{for some $r>0$}.
\end{equation}
We must prove that $r\le 1$, and that $r=1$ if and only if $g=f$.

Let $G:\D\to\Omega$ be the unique holomorphic map with $\Re G=g$ and $G(0)=F(0)$.
In view of the Cauchy--Riemann equations, condition \eqref{eq:dfdg} implies
\begin{equation}\label{eq:dFdG}
	G'(0) = r F'(0),
\end{equation}
where the prime denotes the complex derivative. It follows that the map $(F(z)-G(z))/z$ 
is holomorphic on $\D$, and its value at $z=0$ equals 
\begin{equation}\label{eq:FGprime}
	\lim_{z\to 0} \frac{F(z)-G(z)}{z} = F'(0) - G'(0) = (1-r)F'(0).
\end{equation}

The bounded harmonic map $g:\D\to \B^3$ has a nontangential boundary value 
at almost every point of the circle $\T=b\D$. Since
the Hilbert transform is an isometry on the Hilbert space $L^2(\T)$, the same is true for 
its holomorphic extension $G$ (see Garnett \cite{Garnett2007}).

Denote by $\langle\cdot,\cdot\rangle$ the complex bilinear form on $\C^n$Given by
$
	\langle z,w\rangle = \sum_{i=1}^n z_i w_i
$
for $z,w\in\C^n$. Note that on vectors in $\R^n$ this is the Euclidean inner product.
For each $z=e^{\imath t}\in b\D$ the vector $f(z)\in b\B^3$ is the unit normal vector
to the sphere $b\B^3$ at the point $f(z)$. Since $\B^3$ is strongly convex and $f$ is real-valued, 
we have that
\begin{equation}\label{eq:convexity}
	\Re\, \big\langle F(z)-G(z), f(z) \big\rangle = 
	\big \langle f(z)-g(z), f(z) \big\rangle \ge 0 \quad {a.e.}\ z\in b\D,
\end{equation}
and the value is positive for almost every $z\in b\D$ if and only if $g\ne f$.
It is at this point that strong convexity of the ball $\B^3$ is used in an essential way.

We now consider the map $\tilde f$ on the circle $b\D$Given by
\begin{equation}\label{eq:tildef}
	\tilde f(z) = z |1+\bar a z|^2 f(z),\quad\ |z|=1.
\end{equation}
An explicit calculation, taking into account $z\bar z=1$, shows that
\begin{equation}\label{eq:tildef2}
	\tilde f(z) =
	\left(
	\begin{matrix}
		\frac{c}{2} \left(1+a^2 + 4(\Re a)z + (1+\bar a^2)z^2 \right) \cr\cr
		\frac{c}{2} \left(\imath(1-a^2) + 4(\Im a)z + \imath (\bar a^2-1)z^2 \right)  \cr\cr
		p\, (z+a)(1+\bar a z)
	\end{matrix}
	\right).
\end{equation}
We extend $\tilde f$ to all $z\in \C$ by letting it equal the quadratic holomorphic polynomial map on the
right hand side above. Since $|1+\bar a z|^2>0$ for $z\in\overline\D$, \eqref{eq:convexity} implies
\begin{eqnarray*}
	h(z) &:=& \Re\, \big\langle F(z)-G(z), |1+\bar a z|^2 f(z)\big\rangle \\
	       &=&  \big\langle f(z)-g(z), |1+\bar a z|^2 f(z) \big\rangle \ge 0 \ \ \ {a.e.}\ z\in b\D,
\end{eqnarray*}
and $h>0$ almost everywhere on $b\D$ if and only if $g\ne f$. From \eqref{eq:tildef} we see that
\begin{equation}\label{eq:h}
	h(z) = \Re \, \left\langle \frac{F(z)-G(z)}{z}, \tilde f(z)\right\rangle \ \  {a.e.}\ z\in b\D
\end{equation}
Since the maps $(F(z)-G(z))/z$ and $\tilde f(z)$ are holomorphic on $\D$, 
the formula \eqref{eq:h} provides an extension of $h$ from $b\D$ to a nonnegative harmonic function 
on $\D$ which is positive on $\D$ unless $f=g$. Inserting the value \eqref{eq:FGprime} into 
\eqref{eq:h} gives
\[
	h(0) =  \Re\, \big\langle F'(0) - G'(0) ,\tilde f(0) \big\rangle
	       =  (1-r) \, \Re\, \big\langle F'(0) ,\tilde f(0)\big\rangle  \ge 0,
\]
with equality if and only if $f=g$.  Applying this argument to the linear map
$g(z)=f(0)+r df_0(z)$ $(z\in\D)$ for a small $r>0$ we get $\Re\, \langle F'(0) ,\tilde f(0)\rangle >0$.
It follows that $r\le 1$, with equality if and only if $g=f$.
\end{proof}

%
%
\begin{remark}
\label{rem:stationary}
The main point in the above proof is that 
a complexification of a conformal proper affine disc in $\B^n$ 
is a {\em stationary disc} in the tube  $\Tcal_{\B^n}=\B^n\times \imath \R^n$.
In Lempert's terminology from \cite{Lempert1981}, a proper holomorphic disc $F:\D\to\Omega$ 
in a smoothly bounded convex domain $\Omega\subset \C^n$,
extending continuously to $\overline \D$, 
is a stationary disc if, denoting by $\nu:b\D\to\C^n$ the unit normal vector field to $b\Omega$ 
along the boundary circle $F(b\D)\subset b\Omega$, there is a positive 
continuous function $q>0$ on $b\D$ such that the function $z \,q(z)\overline{\nu(z)}$
extends from the circle $|z|=1$ to a holomorphic function $\tilde f(z)$ on $\D$.
Lempert showed in \cite{Lempert1981} that every stationary disc $F$ in a bounded strongly convex domain
is the unique Kobayashi extremal disc through the point $F(a)$ in the tangent direction $F'(a)$ 
for every $a\in\D$. In our case, a suitable holomorphic function $\tilde f$ is given by \eqref{eq:tildef} 
and \eqref{eq:tildef2}. Lempert's theory also works on tubes over bounded strongly convex domains  
(see Jarnicki and Pflug \cite[Sect. 11.1]{JarnickiPflug2013}); 
however, our proof of Theorem \ref{th:main} does not depend on this information.
\end{remark}

%
%
\section{Proof of Theorem \ref{th:main2}} \label{sec:proofmain2}

Precomposing the given harmonic map $f:\D\to\B^n$ in Theorem \ref{th:main2}
by a holomorphic automorphism of the disc $\D$, we see that it suffices to prove 
the estimate \eqref{eq:S2} for $z=0$. 

Assume first that $f:\D\to \R$ is a harmonic function on $\D$.
Let $F(z)= a_0 + a_1z+a_2 z^2+\ldots$ be the holomorphic function on $\D$ with $\Re F=f$ and $F(0)=f(0)\in\R$. Writing $z=r e^{\imath t}$ with $0\le r<1$ and $t\in\R$, we have that
\begin{eqnarray*}
	f(r e^{\imath t})^2 &=& \frac{1}{4} \big(a_0+a_1 r e^{\imath t}+ r^2 e^{2\imath t}+\cdots +
	a_0 + \bar a_1 r e^{-\imath t}+ \bar a_2 r^2 e^{-2\imath t}+\cdots\big)^2  \\
	&=& a_0^2 + \frac12 \sum_{k=1}^\infty r^{2k}|a_k|^2 + \cdots,
\end{eqnarray*}
where each of the remaining terms in the series contains a power $e^{m\imath t}$ 
for some $m\in \Z\setminus\{0\}$.
Integrating around the circle $|z|=r$ for $0<r<1$ annihilates all such terms and yields
\[
	\int_0^{2\pi} f(r e^{\imath t})^2 \frac{dt}{2\pi} = a_0^2 + \frac12 \sum_{k=1}^\infty r^{2k}|a_k|^2.
\]
Clearly, $a_0=f(0)$. Writing $z=x+\imath y$, we have that $a_1=F'(0)=F_x(0)=f_x(0)-\imath f_y(0)$
by the Cauchy--Riemann equations. Therefore,
\[
	a_0^2=f(0)^2, \qquad |a_1|^2 = f_x(0)^2 + f_y(0)^2 = |\nabla f(0)|^2,
\]
and hence
\begin{equation}\label{eq:a0a1}
	\int_0^{2\pi} f(r e^{\imath t})^2 \frac{dt}{2\pi} = |f(0)|^2 + \frac12  |\nabla f(0)|^2 r^2
	+  \frac12 \sum_{k=2}^\infty r^{2k}|a_k|^2.
\end{equation}

Suppose now that $f=(f_1,\ldots, f_n):\D\to \B^n$ is a harmonic map. Then,
$\sum_{j=1}^n f_j(r e^{\imath t})^2 <1$ for all $0\le r<1$ and $t\in\R$. Integrating this inequality
and taking into account the identity \eqref{eq:a0a1} for each component $f_j$ of $f$Gives
\[	
	\int_0^{2\pi} |f(r e^{\imath t})|^2 \frac{dt}{2\pi}  = |f(0)|^2 + \frac12  |\nabla f(0)|^2 r^2 +
	 \frac12 \sum_{k=2}^\infty r^{2k}|a_k|^2 <1.
\]
Letting $r$ increase to $1$Gives $|f(0)|^2 + \frac12  |\nabla f(0)|^2 \le 1$, with equality
if and only if all higher order coefficients in the Fourier expansion of $f$ vanish. The latter holds if
and only if $f$ is a linear disc. This gives the estimate \eqref{eq:S2}. 

Note that \eqref{eq:S2} holds if the $L^2$-Hardy norm of $f$ is at most $1$.
This does not necessarily imply that there is a harmonic disc in $\B^n$ reaching equality in
\eqref{eq:S2}. However, equality is reached if $f(0)$ is orthogonal to the $2$-plane $df_0(\R^2)$.
In this case we may assume that $f(0)=(0,0,p,0\ldots,0)$ for some
$0\le p<1$ and $df_0(\R^2)=\R^2\times \{0\}^{n-2}$. The affine disc
\[
	\Sigma = \bigl\{(x,y,p,0,\ldots,0):x^2+y^2< 1-p^2\bigr\}
\]
of radius $c=\sqrt{1-p^2}$ is then orthogonal to $f(0)$, proper in $\B^n$, and its conformal linear 
parameterization $f$ has gradient of size $c\sqrt 2$ at the origin, 
so $|f(0)|^2+ \frac12 |\nabla f(0)|^2=p^2+c^2=1$.
(Compare with \eqref{eq:Sigmaf2} and \eqref{eq:extremal}.)
This completes the proof of Theorem \ref{th:main2}.

We now show by examples that the inequality \eqref{eq:S} fails in general for some nonconformal
harmonic maps, and even for harmonic diffeomorphisms of the disc.

%
%
\begin{example}\label{ex:failure}
Let $U$ be the harmonic function on the disc $\D$Given by
\begin{equation}\label{eq:U}
	U(z) = \Im\, \frac{2}{\pi} \log \frac{1+z}{1-z} =  \frac{2}{\pi} \arctan\frac{2y}{1-x^2-y^2}.
\end{equation}
This is the extremal harmonic function whose boundary value equals $+1$ on the upper unit semicircle
and $-1$ on the lower semicircle, and we have that $\nabla U(0) =  \frac{4}{\pi} (0,1)$
and $|\nabla U(0)| = \frac{4}{\pi}$. For every $c\in\R$ the harmonic map
\[
	f(z)=\frac{1}{\sqrt{1+|c|^2}}  \bigl( c+\imath U(z)\bigr),\quad \ z\in\D
\]
clearly takes the unit disc into itself. For $c=1$ we have $f(0)=\frac{1}{\sqrt2}$,
$\nabla f(0)=\frac{2\sqrt{2}}{\pi} \left(\begin{matrix}0& 0 \cr 0& 1 \end{matrix}\right)$,
\[
	|\nabla f(0)|=\frac{2\sqrt{2}}{\pi} \approx 0.9, \quad
	\sqrt{2}\left(1-|f(0)|^2\right) = \frac{\sqrt{2}}{2} \approx 0.7.
\]
Hence, the inequality \eqref{eq:S} fails in this example. On the other hand, 
$\sqrt{2}\sqrt{1-|f(0)|^2}=1$, so the inequality \eqref{eq:S2} holds, as it should
by Theorem \ref{th:main2}.

With some more effort we can show that the inequality \eqref{eq:S} fails for harmonic diffeomorphisms 
of the unit disc onto itself. Consider the sequence  $\varphi_n$ $(n\in\N)$ 
of orientation-preserving homeomorphisms of the interval $[0,2\pi]$ onto itself, defined by
\[
	\varphi_n(t) =\left\{
                        \begin{array}{ll}
                          \frac{\pi}{2\pi-1/n}t , & \hbox{if $t\in [0,2\pi-1/n]$;} \\
                          2 \left(\pi -n \pi ^2\right)+n \pi  t, & \hbox{if $t\in [2\pi-1/n,2\pi]$.}
                        \end{array}
                      \right.
\]
Let $\phi_n:\T\to\T$ be the associated sequence of homeomorphisms of the circle $\T=b\D$
given by $\phi_n(e^{\imath t}) = e^{\imath \varphi_n(t)}$ for $t\in [0,2\pi]$. Denote by 
\[
	f_n(z) = P[\phi_n](z) = \frac{1}{2\pi} \int_0^{2\pi} 
	\frac{1-|z|^2}{|e^{\imath t}-z|^2} \, \phi_n(e^{\imath t})  dt,\quad z\in\D
\]
the Poisson extension of $\phi_n$. By Rad{\'o}--Kneser--Choquet theorem 
(see \cite[Sect.\ 3.1]{Duren2004}), $f_n$ is a harmonic 
diffemorphism of $\D$ for every $n\in\N$. As $n\to\infty$, the sequence $f_n$ 
converges uniformly on compacts in $\D$ to the harmonic map $f=P[\phi_0](z)$, 
where $\phi_0(e^{\imath t}) =\lim_{n\to\infty} \phi_n(e^{\imath t}) = e^{\imath t/2}$
for $t\in [0,2\pi)$. Further, 
\[
	\lim_{n\to \infty}\frac{|\nabla f_n(0)|}{1-|f_n(0)|^2}=\frac{|\nabla f(0)|}{1-|f(0)|^2}.
\]
A calculation shows that 
\[
	\frac{1}{\sqrt2} \frac{|\nabla f(0)|}{1-|f(0)|^2}=  \frac{\sqrt{|A|^2+|B^2|}}{1-|C|^2},
\]
where
\[
	A= \frac{1}{\pi}\int_0^{2\pi} e^{\frac{\imath t}{2}} \cos t \,dt = -\frac{4\imath}{3 \pi}, 
	\qquad 
	B=\frac{1}{\pi}\int_0^{2\pi} e^{\frac{\imath t}{2}} \sin t\, dt = \frac{8}{3 \pi}, 
\]
and
\[
	C=\frac{1}{2\pi}\int_0^{2\pi} e^{\imath t/2} dt=\frac{2\imath}{\pi}.
\]
Hence, 
\[
	\frac{1}{\sqrt2} \frac{|\nabla f(0)|}{1-|f(0)|^2}
	= \frac{2 \sqrt{10}}{3\pi \left(1-\frac{4}{\pi ^2}\right)}\approx 1.1.
\]
This shows that \eqref{eq:S} fails for harmonic diffeomorphisms of the unit disc onto itself.
\end{example}

\begin{example}\label{ex:square}
Let $U(x,y)$ be the function \eqref{eq:U}. The harmonic map $f(x,y)=(U(y,x),U(x,y))$ takes
the disc $\D$ onto the square $P=\{(x,y)\in\R^2: |x|<1,\ |y|<1\}$ and 
$df_0(0,0)=\frac{4}{\pi}\Id$. In particular, $f$ is conformal at $(0,0)$ and 
$\|df_0\|=4/\pi\approx 1.27$. On the other hand, a conformal diffeomorphism
of $\D$ onto $P$ mapping the origin to itself has the derivative at the origin
of absolute value $\approx 1.08$. Hence, the Schwarz--Pick lemma in 
Theorem  \ref{th:SP} fails for maps from the disc to more general (convex) domains in $\C$.
\end{example}

\begin{problem}\label{prob:disc}
Assume that $D\subsetneq \R^2$ is a simply connected domain such that, for some point $p\in D$,
the supremum of the norm $\|df_0\|$ of the differential of $f$ at $0\in\D$ 
over all harmonic maps $f:\D\to D$ with $f(0)=p$ which are conformal
at $0$ is reached by a conformal diffeomorphism of $\D$ onto $D$.
Does it follow that $D$ is a disc?
\end{problem}

%
%
\section{A Schwarz--Pick lemma for quasiconformal harmonic maps}\label{sec:Beltrami}
In this section we apply the Schwarz--Pick lemma for harmonic self-maps of the disc,
given by Theorem \ref{th:SP}, to provide an estimate of the gradient of a harmonic map
$f:\D\to\D$ in terms of its {\em second Beltrami coefficient}
\begin{equation}\label{eq:SBC}
	\omega(z) = \frac{\overline{(f_{\bar z})}}{f_z},\quad\ z\in\D.
\end{equation}
Here, $f_z=\frac12 (f_x-\imath f_y)$ and $f_{\bar z}=\frac12 (f_x+\imath f_y)$. 
If the map $f$ is harmonic then $\omega$ is a holomorphic function (see \eqref{eq:hgomega}). 
This is not the case for the Beltrami coefficient $\mu$ from the
Beltrami equation $f_{\bar z} = \mu (z) f_z$. The number $|\mu(z)|=|\omega(z)|$
measures the dilatation of $df_z$; in particular, $\mu(z)=\omega(z)=0$
if and only if $f$ is conformal at $z$. We refer to Ahlfors  \cite{Ahlfors2006},
Duren \cite{Duren2004}, Lehto and Virtanen \cite{LehtoVirtanen1973}, and 
Hengartner and Schober \cite{HengartnerSchober1986} for background on the theory
of quasiconformal maps.

The main question is to find the optimal estimate on 
$|\nabla f(0)|$ for a harmonic map $f:\D\to\D$ with $f(0)=0$ and with a given value
of $|\omega(0)|=|\mu(0)|$.  A related problem was studied by Kovalev and Yang 
\cite{KovalevYang2020} and Brevig et al.\ \cite{BrevigOrtegaSeip2021}, where 
the reader can find references to earlier works. Here we prove the following result.

%
%
\begin{theorem}\label{th:Beltrami}
Assume that $f$ is an orientation preserving harmonic map of the unit disc into itself, and let
$\omega(z)$ denote its second Beltrami coefficient \eqref{eq:SBC}.
Then we have the inequality
\begin{eqnarray*}
	\frac{1}{\sqrt2} |\nabla f(z)| &\le& \frac{2\Re (\omega(z) f(z)^2)}{(1-|\omega(z)|^2)(1-|z|^2)} +
	\frac{1+|\omega(z)|^2}{1-|\omega(z)|^2}\frac{1-|f(z)|^2}{1-|z|^2},\quad z\in \D.
\end{eqnarray*}
\end{theorem}

If $f$ is conformal at a point $z$, i.e. $\omega(z) = 0$, this coincides
with the Schwarz--Pick inequality \eqref{eq:SP} in Theorem \ref{th:SP}.
The estimate is not sharp in general. For example, if $f(0)=0$ then the left hand side 
is at most $1$ by Theorem \ref{th:main2}, but the right hand side is at least $1$
and equals $1$ only if $f$ is conformal at $0$. Hence, the inequality trivially holds 
in this case. However, it is nontrivial at points where $f(z)\ne 0$.

\begin{proof}
It suffices to prove the inequality in the theorem for $z=0$. For other points, we obtain it 
replacing $f$ by $f\circ \phi_z$ for $\phi_z\in \Aut(\D)$. However, we cannot reduce to 
the case $f(0)=0$ since postcompositions by automorphisms of $\D$ are not allowed.
The main idea is to construct from $f$ a new harmonic map $\tilde f:\D\to\D$ which is conformal at $0$,
to which we then apply the Schwarz--Pick lemma given by Theorem \ref{th:SP}.

Let us write $f = g + \overline{h}$ where $g$ and $h$ are holomorphic functions on $\D$. Then, 
\begin{equation}\label{eq:hgomega}
	f_z(z)=g'(z), \quad f_{\bar z}(z)=\overline{h'(z)}, \quad \omega(z)=h'(z)/g'(z). 
\end{equation}
We see in particular that the second Beltrami coefficient $\omega$ is holomorphic. It follows that 
\begin{equation}\label{eq:gradcomplex}
	|\nabla f|^2 = |f_x|^2+|f_y|^2 = 2 \left(|f_z|^2+|f_{\bar z}|^2\right)
	= 2 \left( |g'|^2 + |h'|^2\right) = 2|g'|^2 (1+|\omega|^2).	
\end{equation}
Since $f$ is sense preserving, we have that $|g'(z)|\ge |h'(z)|$. Let $a=g'(0)$ and $b= h'(0)$. 
The Cauchy--Schwarz inequality shows that the complex harmonic function
\[
	 \tilde f(z)  = \frac{\bar a f-\bar b \bar f}{\sqrt{|a|^2+|b|^2}},\quad\ z\in \D
\]
maps the unit disc into itself. We have $\tilde f = \tilde g + \overline{ \tilde h}$, where
\[
	\tilde g=\frac{\bar a g-\bar b h}{\sqrt{|a|^2+|b|^2}}\quad
	\text{and}\quad
	\tilde h=\frac{ a h-b g}{\sqrt{|a|^2+|b|^2}}
\]
are holomorphic functions. Since
\begin{equation}\label{eq:tildehzero}
	\tilde h'(0)= \frac{ a h'(0)- b g'(0)}{\sqrt{|a|^2+|b|^2}}=0,
\end{equation}
$\tilde f$ is conformal at $z=0$. Our Schwarz--Pick lemma (see Theorem \ref{th:SP}) gives
\begin{equation}\label{eq:Sbis3}
	\sqrt{2}^{-1} |\nabla \tilde f(0)|  \le  1-|\tilde f(0)|^2.
\end{equation}
Taking into account \eqref{eq:gradcomplex} and \eqref{eq:tildehzero} we have that 
\[
	\begin{split}
	\sqrt{2}^{-1} |\nabla \tilde f(0)| 
	&= \sqrt{|\tilde f_z(0)|^2+|\tilde f_{\bar z}(0)|^2}
	=\sqrt{|\tilde g'(0)|^2+|\tilde h'(0)|^2}\\
	&= |\tilde g'(0)|= \frac{|a|^2-|b^2|}{\sqrt{|a|^2+|b|^2}} 
	= \frac{|g'(0)|^2-|h'(0)|^2}{\sqrt{|g'(0)|^2+|h'(0)|^2}}.
	\end{split}
\]
Together with \eqref{eq:Sbis3} this gives the estimate
\begin{eqnarray*}
	\frac{|g'(0)|^2-|h'(0)|^2}{\sqrt{|g'(0)|^2+|h'(0)|^2}}
	&\le&   1-\frac{|g'(0)\overline{f(0)}-h'(0)f(0)|^2} {|g'(0)|^2+|h'(0)|^2} \\
	&\le& 
	\frac{2\Re\left(g'(0)\overline{h'(0)}\overline{f(0)^2}\right) }{|g'(0)|^2+|h'(0)|^2} + 1-|f(0)|^2.
\end{eqnarray*}
In view of \eqref{eq:hgomega}, this inequality can be written in the form
\[
	\frac{1-|\omega(0)|^2}{\sqrt{1+|\omega(0)|^2}} \, |g'(0)|\le
	\frac{2\Re (\omega(0) f(0)^2)}{1+|\omega(0)|^2}+1-|f(0)|^2.
\]
From \eqref{eq:gradcomplex} we see that 
\[
	|g'(0)|= \frac{|\nabla f(0)|}{\sqrt2 \sqrt{1+|\omega(0)|^2}}.  
\]
Inserting this into the expression on the left hand side of the previous inequality gives
\[
	 \frac{|\nabla f(0)|}{\sqrt2} \frac{1-|\omega(0)|^2}{1+|\omega(0)|^2} 
	 \le \frac{2\Re (\omega(0) f(0)^2)}{1+|\omega(0)|^2}+1-|f(0)|^2.
\]
which is clearly equivalent to
\[
	 \frac{1}{\sqrt2} |\nabla f(0)| \le \frac{2\Re (\omega(0) f(0)^2)}{1-|\omega(0)|^2}
	 + \frac{1+|\omega(0)|^2}{1-|\omega(0)|^2} \left(1-|f(0)|^2\right).
\]
This completes the proof.
\end{proof}

%
%
\section{An intrinsic pseudometric defined by conformal harmonic discs} \label{sec:hyperbolicity}
In this section we introduce an intrinsic Finsler pseudometric $g_D$ on any domain $D$ in $\R^n,\ n\ge 3$,  
and more generally on any Riemannian manifold of dimension at least three, in terms of conformal
minimal discs $\D\to D$. The definition is modelled on Kobayashi's definition of his pseudometric 
on complex manifolds, which uses holomorphic discs. The pseudometric $g_D$ and 
the associated pseudodistance $\rho_D:D\times D\to\R_+$ are the largest ones 
having the property that any conformal harmonic map $M\to D$ from a hyperbolic
conformal surface with the  Poincar\'e metric is metric- and distance-decreasing.
On the ball $\B^n$, $g_{\B^n}$ coincides with the Cayley--Klein metric 
(see Theorem \ref{th:metricsagree}). 
The same definition of $g_D$ applies in any Riemannian manifold of dimension at least three;
see Remark \ref{rem:Riemannianmanifolds}.
This provides the basis for hyperbolicity theory of domains in Euclidean spaces and,
more generally, of Riemannian manifolds, in terms of minimal surfaces. 

We begin by introducing a Finsler pseudometric on the bundle of $2$-planes over 
a domain $D\subset\R^n$, analogous to the metric $\Mcal$ on the ball; see \eqref{eq:Mcal}. 

A {\em conformal frame} in $\R^n$ is a pair $(\bu,\bv)\in \R^n\times \R^n$ 
such that $|\bu|=|\bv|$ and $\bu\,\cdotp \bv=0$. We denote by $\CF_n$ the space of all
conformal frames on $\R^n$, including $(\zero,\zero)$. Given a domain $D\subset \R^n$, 
let $\CH(\D,D)$ denote the space of conformal harmonic maps $\D\to D$ 
(i.e., \eqref{eq:conformal} holds at every point of $\D$). 
Define the function $\Mcal_D: D\times \CF_n\to\R_+$ by 
\begin{equation}\label{eq:minmetric}
	\Mcal_D(\bx,(\bu,\bv)) = \inf \bigl\{ 1/r  : 
	\exists f\in \CH(\D,D),\ f(0)=\bx,\ f_x(0) = r\bu,\ f_y(0) = r\bv\}.
\end{equation}
Clearly, $\Mcal_D$ is homogeneous and rotation-invariant, in the sense that 
for any $c\in\R$ and orthogonal rotation $R$ in the $2$-plane $\Lambda=\span\{\bu,\bv\}$ we have 
for every $\bx\in D$ that 
\[
	\Mcal_D(\bx,(c\bu,c\bv)) = |c| \Mcal_D(\bx,(\bu,\bv)),\quad\ 
	\Mcal_D(\bx,(R\bu,R\bv)) = \Mcal_D(\bx,(\bu,\bv)).
\]
Thus, $\Mcal_D$ is determined by its values on unitary conformal frames $(\bu,\bv)$
with $|\bu|=|\bv|=1$, and hence on $D\times  G_2(\R^n)$ where $G_2(\R^n)$ is the Grassmann 
manifold of $2$-planes in $\R^n$. Precisely, for a $2$-plane $\Lambda\in  G_2(\R^n)$ 
we set $\Mcal_D(\bx,\Lambda)=\Mcal_D(\bx,(\bu,\bv))$, where $(\bu,\bv)$ is any
unitary conformal frame spanning $\Lambda$. Note that
\begin{equation}\label{eq:McalL}
	\Mcal_D(\bx,\Lambda) = \inf \bigl\{ 1/\|df_0\|  : f\in \CH(\D,D),\ f(0)=\bx,\ df_0(\R^2)=\Lambda\}.
\end{equation}
By shrinking the disc $\D$ and using rotations and translations on $\R^n$, 
we see that the function $\Mcal_D$ is upper semicontinuous on $D\times \CF_n$. 
Obviously, $\Mcal_{\R^n}\equiv 0$. On the ball $\B^n$, $\Mcal_{\B^n}(\bx,\Lambda)$ is given by 
\eqref{eq:Mcal} according to Theorem \ref{th:main}. 

We also introduce a Finsler pseudometric $g_D:D\times \R^n\to\R_+$, called the
{\em minimal metric} on $D$, whose value at a point $\bx\in D$ on a tangent vector 
$\bu\in T_\bx D=\R^n$ is given by 
\begin{eqnarray}\label{eq:gmin1}
	g_D(\bx,\bu) &=& \inf\bigl\{1/r>0: \exists f\in\CH(\D,D),\ f(0)=\bx,\ f_x(0)=r\bu\bigr\} \\
	&=& |\bu| \,\cdotp \inf\big\{ \Mcal_D(\bx,\Lambda): \Lambda\in  G_2(\R^n),\ \bu\in\Lambda \big\}.
	\label{eq:gmin2}
\end{eqnarray}
It follows that every conformal harmonic map $f:\D\to D$ satisfies 
\begin{equation}\label{eq:metricdecreasing2}
	g_D\bigl(f(z),df_z(\xi)\bigr) \ \le\ \Pcal(z,\xi) = \frac{|\xi|}{1-|z|^2},  \quad\ z\in \D,\ \xi\in\R^2,
\end{equation}
and $g_D$ is the biggest pseudometric on $D$ with this property.
For $z=0$ this follows directly from the definition, and for any other point $z\in \D$ 
we precompose $f$ by a conformal automorphism of $\D$ mapping $0$ to $z$. 
The same holds if $\D$ is replaced by any hyperbolic conformal surface
(see the proof of Theorem \ref{th:metricdecreasing}).

By integrating $g_D$ we get the {\em minimal pseudodistance} 
$\rho_D:\Omega\times \Omega\to \R_+$:  
\begin{equation}\label{eq:rho}
	\rho_D(\bx,\by)= \inf_\gamma \int_0^1 g_D(\gamma(t),\dot\gamma(t)) \, dt,
	\quad\ \bx,\by\in\Omega.
\end{equation}
The infimum is over all piecewise smooth paths $\gamma:[0,1]\to\Omega$ with $\gamma(0)=\bx$
and $\gamma(1)=\by$. Obviously, $\rho_\Omega$ satisfies the triangle inequality, but it need not
be a distance function. In particular, $\rho_{\R^n}$ vanishes identically.

There is another natural procedure to obtain the pseudodistance $\rho_D$ in \eqref{eq:rho}, 
which is motivated by Kobayashi's definition of his pseudodistance on 
complex manifolds in \cite{Kobayashi1967}. 
Fix a pair of points $\bx,\by\in D$. To any finite chain of conformal harmonic 
discs $f_i:\D\to D$ and points $a_i\in \D$ $(i=1,\ldots,k)$ such that 
\begin{equation}\label{eq:chain}
	f_1(0)=\bx,\quad f_{i+1}(0)=f_i(a_i)\ \text{for $i=1,\ldots,k-1$},\quad f_k(a_k)=\by
\end{equation}
we associate the number 
\[
	\sum_{i=1}^k \frac{1}{2}\log \frac{1+|a_i|}{1-|a_i|}\ge 0.
\]
The $i$-th term in the sum is the Poincar\'e distance from $0$ to $a_i$ in $\D$.
The pseudodistance $\rho_D(\bx,\by)$ is defined to be the infimum of the numbers
obtained in this way. The proof that the two definitions yield the same result
is similar to the one given for the Kobayashi pseudodistance by Royden \cite[Theorem 1]{Royden1971};
see \cite[Theorem 3.1]{DrinovecForstneric2021X} for the details.

The following proposition says that the minimal pseudodistance $\rho_D$Gives as upper bound  
for growth of conformal minimal surfaces in the domain $D$.

%
%
\begin{proposition}\label{prop:distancedecreasing}
Every conformal harmonic map $M\to D$ from a hyperbolic conformal surface is distance-decreasing
in the Poincar\'e distance on $M$ and the pseudodistance $\rho_D$, and 
$\rho_D$ is the biggest pseudodistance on $D$ for which this holds. 
\end{proposition}

\begin{proof} 
Let $M$ be a hyperbolic conformal surface and $h:\D\to M$ be a conformal universal covering.
Choose a conformal harmonic map $f:M\to  D$ and a pair of points $p,q\in M$. 
Let $a,b\in\D$ be such that $h(a)=p$ and $h(b)=q$. Precomposing $h$ by an automorphism 
of the disc, we may assume that $a=0$. Then, $g:=f\circ h:\D\to D$ is a conformal harmonic 
disc with $g(0)=f(p)$ and $g(b)=f(q)$, and it follows from the definition of $\rho_D$ that 
\[
	\rho_D(f(p),f(q)) \ =\ \rho_D(g(0),g(b)) \ \le\ \frac12 \log\frac{1+|b|}{1-|b|}. 
\]
The infimum of the right hand side over all points $b\in\D$ with $h(b)=q$ equals 
the Poincar\'e distance between $p$ and $q$ in $M$, so we see that $f$ is distance-decreasing.

Suppose now that $\tau$ is a pseudodistance on $D$ such that every conformal harmonic
map $\D\to D$ is distance-decreasing with the Poincar\'e metric on $\D$. Let 
$f_i:\D\to D$ and $a_i\in \D$ for $i=1,\ldots,k$ be a chain as in \eqref{eq:chain} 
connecting the points $\bx,\by\in D$. Then, 
\[
	\tau(\bx,\by) \ \le\ \sum_{i=1}^k \tau(f_i(0),f_i(a_i)) \ \le\  
	\sum_{i=1}^k \frac{1}{2}\log \frac{1+|a_i|}{1-|a_i|}.
\]
Taking the infimum over all such chains gives $\tau(\bx,\by)\le \rho_D(\bx,\by)$.
\end{proof}

We have already observed that on the ball $\B^n$ $(n\ge 3)$, the Finsler metric 
$\Mcal_{\B^n}$ is given by \eqref{eq:Mcal}. 
From \eqref{eq:min} and \eqref{eq:gmin2} it follows that   
$g_{\B^n}$ equals the Cayley--Klein metric $\CK$ \eqref{eq:CK}:

%
%
\begin{theorem}\label{th:metricsagree}
On the ball $\B^n,\ n\ge 3,$ we have that
\[
	g_{\B^n} = \CK,\qquad \rho_{\B^n} = \dist_\CK.
\]
\end{theorem}

%
%
In 1985, Hilbert \cite{Hilbert1885} defined a metric on any convex domain in $\RP^n$ 
that generalizes the Cayley--Klein metric on the ball.
Hilbert metrics are examples of projectively invariant metrics 
which have been studied by many authors; see the surveys by Kobayashi
\cite{Kobayashi1977,Kobayashi1984} and Goldman \cite{Goldman2019}.
Kobayashi discussed the analogy between his metric and Hilbert's metric in \cite{Kobayashi1977}.
An explicit connection was established by Lempert who wrote about it
in \cite{Lempert1987}, and then in \cite[Theorem 3.1]{Lempert1993} proved that
the Hilbert metric $\Hcal_D$ on any bounded convex domain $D\subset \R^n$ is the restriction to $D$
of the Kobayashi metric on the elliptic tube $D^*\subset D\times \imath\R^n\subset \C^n$ 
obtained as follows (see \cite[p.\ 441]{Lempert1993}). Every affine line segment $L\subset D$ 
with endpoints on $bD$ is the diameter of a unique complex disc in $D\times \imath\R^n$,
and $D^*$ is the union of all such discs. The elliptic tube over the ball $\B^n$ is the complex ball 
$\B^n_\C$, and the metric $g_{\B^n}$ agrees with the Hilbert metric 
$\Hcal_{\B^n}=\CK$ according to Theorem \ref{th:metricsagree}.

We now give an example of a strongly convex domain $D \subset\R^3$ (an ellipsoid)  
on which the minimal metric $g_D$ does not coincide with the Hilbert metric $\Hcal_D$. 

%
%
\begin{example}\label{ex:ellipsoid}
Let $(x,y,z)$ be coordinates on $\R^3$. For $a>0$ consider the ellipsoid
\[
	D_a=\Big\{(x,y,z)\in\R^3: x^2+\frac{1}{a^2} (y^2+z^2) <1 \Big\}.
\] 
Note that $D_a\subset \B^3$ if and only if $0<a\le 1$, and $D_1=\B^3$.
We will show that for $0<a<1$ the Hilbert metric on $D_a$ does not agree with the
minimal metric at the origin $\zero\in \R^3$. Since the $x$-axis intersects
$D_a$ in the interval $(-1,1)$, the Hilbert length of the vector $\be_1=(1,0,0)$ equals
$1$. 
Pick a $2$-plane $\Lambda\subset \R^3$ containing the vector $\be_1$. 
Due to rotational symmetry of $D_a$ in the $(y,z)$-coordinates the value 
of $\Mcal_{D_a}(\zero,\Lambda)$ \eqref{eq:McalL} does not depend on 
the choice of $\Lambda$, so we may take $\Lambda = \{z=0\}$.
Let $f=(f_1,f_2,f_3):\D\to D_a$ be a conformal harmonic disc with $f(0)=\zero$ 
and $df_0(\R^2)=\{z=0\}$. Replacing $f$ by $f(\E^{\imath t}z)$ 
for a suitable $t\in \R$Gives $f_x(0)=r\be_1$ and $f_y(0) = \pm r\be_2$ with $r=\|df_0\|>0$.
The projection $h=(f_1,f_2):\D\to\R^2$ maps $\D$ into the ellipse
$E_a=\{x^2+y^2/a^2<1\}$, $h(0)=\zero$, and $h$ is conformal at $0$. For $0<a<1$
we have $E_a\subsetneq\D$. Theorem \ref{th:SP} implies that $r=\|dh_0\|<1$;
equality is excluded since in that case we would necessarily have $h(\D)=\D$.
By a normal families argument we also have that $\sup_f \|df_0\|<1$. It follows
that $\Mcal_{D_a}(\zero,\Lambda)>1$ for every such $\Lambda$, and hence 
$g_{D_a}(\zero,\be_1)=\Mcal_{D_a}(\zero,\Lambda) >1=\Hcal_{D_a}(\zero,\be_1)$ if $0<a<1$.
\end{example}

\begin{problem}
On which bounded convex domains $D\subset\R^n,\ n\ge 3$ (besides the ball) 
does the Hilbert metric coincide with the minimal metric $g_D$?
Is the ball the only such domain?
\end{problem}

Denote by $\Rcal_n$ the Lie group of transformations $\R^n\to\R^n$Generated 
by the orthogonal group $O_n$, translations, and dilatations by positive numbers. 
Elements of $\Rcal_n$ are called {\em rigid transformations} of $\R^n$. 
Postcomposition of any conformal harmonic map $f:M\to\R^n$ by a rigid transformation of $\R^n$ 
is again a conformal harmonic map, and it is well known that $\Rcal_n$ is the largest group 
of diffeomorphisms of $\R^n$ having this property. This gives

\begin{proposition}\label{prop:isometries}
Given a domain $D\subset \R^n,\ n\ge 3,$ and a map $R\in \Rcal_n$, the restriction 
$R|_D:D\to D'=R(D)$ is an isometry of pseudometric spaces $(D,\rho_D)\to (D',\rho_{D'})$.
\end{proposition}

%
%
\begin{remark}\label{rem:Riemannianmanifolds}
The intrinsic pseudometric $g_D$ and the associated pseudodistance $\rho_D$
can be defined in the very same way on an arbitrary Riemannian manifold
$(D,\tilde g)$ of dimension at least three. The Riemannian metric $\tilde g$ determines 
the class of conformal harmonic maps $\D\to D$, which coincide
with conformal minimal discs in $D$. 
\end{remark}

%
%

\smallskip
\noindent{\bf Hyperbolic domains in $\R^n$.} 
We now introduce the notion of (complete) hyperbolic domains in $\R^n$, in analogy with 
Kobayashi hyperbolic complex manifolds.

%
%
\begin{definition} 
A domain $D\subset \R^n$ $(n\ge 3)$ is {\em hyperbolic} if the pseudodistance
$\rho_D$ is a distance function on $D$, and is {\em complete hyperbolic} if 
$(D,\rho_D)$ is a complete metric space. 
\end{definition}

%
%
\begin{example}
\begin{enumerate}[\rm (a)]
\item The ball $\B^n\subset\R^n$ $(n\ge 3)$ is complete hyperbolic. Indeed, the Cayley--Klein metric
\eqref{eq:CK} is complete, so the conclusion follows from Theorem \ref{th:metricsagree}.

\item Every bounded domain $D\subset \R^n$ is hyperbolic. Indeed, if $B$ is a ball containing 
$D$ then $\rho_D(\bx,\by) \ge \rho_B(\bx,\by)$ for any pair $\bx,\by\in D$, 
and $B$ is complete hyperbolic by (a).
However, a bounded domain need not be complete hyperbolic. For example, if $bD$
is strongly concave at $\bp\in bD$, there is a conformal linear disc $\Sigma \subset D\cup\{\bp\}$ 
containing $\bp$, and it is easily seen that $\bp$ is at finite $\rho$-distance from $D$. 

\item The half-space 
$
	\H^n=\{\bx=(x_1,\ldots,x_n)\in\R^n: x_n>0\}
$
is not hyperbolic, and the pseudodistance $\rho_{\H^n}$ vanishes on all planes $x_n=const$.  
However, every point on $b\H^n=\{x_n=0\}$ is at infinite minimal distance from points in $\H^n$
\cite[Lemma 5.2]{DrinovecForstneric2021X}. 
\end{enumerate}
\end{example}

By using the expression for the metric \eqref{eq:Mcal} on the ball we can determine the 
asymptotic rate of growth of the Finsler metric $\Mcal_D$, and hence of the distance function $\rho_D$,
on any bounded strongly convex domain $D\subset \R^n$ with $\Cscr^2$ boundary.
Let $\delta=\delta(\bx)=1-|\bx|$ denote the distance from a point $\bx\in \B^n\setminus\{0\}$ 
to the sphere $b\B^n$, and let $\Lambda\subset \R^n$ be a 2-plane forming an angle $\theta$ with $\bx$.
As $\bx$ approaches the sphere radially, we have 
\[
	\Mcal_{\B^n}(\bx,\theta):= 
	\Mcal_{\B^n}(\bx,\Lambda) \approx \frac{\sqrt{\cos^2\theta + 2\delta\sin^2\theta}}{2\delta},
\]
in the sense that the quotient of the two sides converges to $1$ as $\delta\to 0$. 
In particular,
\[
	\Mcal_{\B^n}(\bx,\pi/2) \approx \frac{1}{\sqrt{2\delta}}, \qquad 
	\Mcal_{\B^n}(\bx,\theta)\approx \frac{\cos\theta}{2\delta}\ \ \text{for $\theta\in[0,\pi/2)$}.
\]
Assume now that $D\subset \R^n$ is a bounded strongly convex domain with $\Cscr^2$ boundary.
There is a collar $U\subset\R^n$ around $bD$ such that every point $\bx\in U\cap D$
has a unique closest point $\pi(\bx)\in bD$. Comparison with inscribed 
and circumscribed balls to $D$ passing through the point $\pi(\bx)$ shows that
there are constants $0<c<C$ such that 
\begin{equation}\label{eq:estimates}
	c \, \frac{\sqrt{\cos^2\theta + 2\delta\sin^2\theta}}{2\delta} 
	\ \le\ \Mcal_D(\bx,\Lambda) \ \le\ 
	C \frac{\sqrt{\cos^2\theta + 2\delta\sin^2\theta}}{2\delta} 
\end{equation}
for $\bx\in U\cap D$, where $\delta=|\bx-\pi(\bx) |=\dist(\bx,bD)$ and $\theta$ is the angle
between the 2-plane $\Lambda$ and the normal vector $N_{\bx}=\delta^{-1}(\pi(\bx)-\bx)$
to $bD$ at $\pi(\bx)\in bD$. The upper bound uses comparison with inscribed balls,
so it holds on any domain with $\Cscr^2$ boundary, while the lower bound uses comparison 
with circumscribed ball, and hence it depends on strong convexity of $D$. 
These estimates are analogous to the asymptotic boundary estimates of the Kobayashi metric 
in bounded strongly pseudoconvex domains in $\C^n$, due to Graham \cite{Graham1975}. 
(There is a large subsequent literature on this subject.)  
These estimates show in particular 
that the distance function $\rho_D$ induced by $\Mcal_D$ is complete, thereby giving

%
%
\begin{theorem}\label{th:complete}
Every bounded strongly convex domain in $\R^n, n\ge 3,$ with $\Cscr^2$ boundary is complete hyperbolic in the minimal metric.
\end{theorem}

%
%
\begin{remark}\label{rem:BDF}
Since the first version of this paper was posted on arXiv in February 2021, 
progress on the subject of minimal hyperbolicity was made 
by Drinovec Drnov\v sek and Forstneri\v c \cite{DrinovecForstneric2021X}.
Besides establishing basic characterizations of (complete) hyperbolicity, they 
proved that a not necessarily bounded convex domain in $\R^n$ is hyperbolic 
if and only it is complete hyperbolic if and only if it does not contain any affine 2-plane
\cite[Theorem 5.1]{DrinovecForstneric2021X}. Furthermore,
every bounded strongly minimally convex domain in $\R^n,\ n\ge 3,$
is complete hyperbolic \cite[Theorem 9.2]{DrinovecForstneric2021X}.
Their proof relies on the lower bound for $\Mcal_\Omega$ (and hence $g_\Omega$)
given by another Finsler pseudometric $\Fcal_\Omega:\Omega\times  G_2(\R^n)\to\R_+$
defined in terms of minimal plurisubharmonic functions (see \cite[Sect.\ 7]{DrinovecForstneric2021X}).
A discussion of this class of domains and functions
can be found in \cite{AlarconDrinovecForstnericLopez2019TAMS} 
and \cite[Chapter 8]{AlarconForstnericLopez2021}. Finally, they established a localization theorem 
for the minimal pseudometric, analogous to the results for the Kobayashi pseudometric 
(see \cite[Sect.\ 8]{DrinovecForstneric2021X}).
\end{remark}

The following problem remains open; an affirmative answer is known for the case when $M$ is 
a plane (see \cite[Lemma 5.2]{DrinovecForstneric2021X}). 

%
%
\begin{problem}
Let $M$ be an embedded minimal surface in $\R^3$. 
Is the minimal distance from $\R^3\setminus M$ to $M$ infinite?
Is the complement of a catenoid in $\R^3$ complete hyperbolic? 
\end{problem}

%
%
\noindent {\bf Extremal minimal discs.} 
Another important and natural question is the following.

\begin{problem}
Let $D\subset\R^n$ be a bounded strongly convex domain with smooth boundary.
Is there is a unique (up to a conformal reparametrization) 
extremal conformal harmonic disc through any given point $\bx\in D$
tangent to a given $2$-plane $\Lambda\in  G_2(\R^n)$ at $\bx$? 
\end{problem}

Theorem \ref{th:main} gives an affirmative answer
on the ball, and this is the only domain for which the answer seems to be known.
By the seminal result of Lempert \cite{Lempert1981,Lempert1987}, the analogous result holds
for the extremal holomorphic discs for the Kobayashi metric in any smoothly bounded strongly 
convex domain $D\subset\C^n$. 

We now describe a condition which implies an affirmative answer to this problem.
It explores a comparison between the Finsler pseudometric $\Mcal_D$ \eqref{eq:minmetric} 
on a domain $D\subset\R^n$ and a Kobayashi-type pseudometric on the tube 
$\Tcal_D=\D\times \imath\R^n \subset\C^n$. 
To this end, we recall a few basic facts from the theory of minimal surfaces;
see \cite[Chapter 2]{AlarconForstnericLopez2021} or \cite{Osserman1986}. 

A holomorphic map $F=(F_1,\ldots,F_n):\D\to \C^n$ satisfying
$\sum_{i=1}F'_i(z)^2=0$ for all $z\in\D$ is called a {\em holomorphic null map}.
The complex cone
\begin{equation}\label{eq:nullcone}
	\nullq^{n-1} = \Bigl\{\bz=(z_1,\ldots,z_n)\in\C^n: \sum_{i=1}^n z_i^2=0\Bigr\}	
\end{equation}
is called the {\em null cone}, and its elements are {\em null vectors}.
Hence, a holomorphic map $F$ is null if and only if the complex derivative $F'(z)$ lies in $\nullq^{n-1}$
for every $z$. It is a basic fact that the real and imaginary parts of a holomorphic null map $M\to \C^n$ 
are conformal harmonic maps $M\to \R^n$; conversely, every conformal harmonic map $\D\to \R^n$ 
from the disc is the real part of a holomorphic null map $\D\to \C^n$. 
(See \cite[Theorem 2.3.4]{AlarconForstnericLopez2021}.)
Given a domain $D\subset\R^n$, we denote by $\HN(\D,\Tcal_D)$ the space of all 
holomorphic null maps $F=(F_1,\ldots,F_n):\D\to \Tcal_D$.
We define a pseudometric on $(\bz,\bw)\in \Tcal_D\times \nullq^{n-1}$ by
\begin{equation}\label{eq:nullmetric}
	\Ncal_D(\bz,\bw) = \inf \bigl\{1/|a| : 
	\exists F\in \HN(\D,\Tcal_D),\ F(0)=\bz,\ F'(0)=a\bw \bigl\}.
\end{equation}
Here, $a$ may be a complex number. Clearly, $\Ncal_D(\bz,\bw)$ is 
bigger than or equal to the Kobayashi pseudonorm of the vector $\bw\in T_\bz(\Tcal_D)$,
since in the definition of the latter one uses all holomorphic discs as opposed to just null discs. 
Note that for each conformal frame $(\bu,\bv)\in \CF_n$ the vectors $\bu \pm \imath\bv \in\C^n$ are null vectors;
conversely, the real and imaginary components of a null vector $\bw\in\nullq^{n-1}$ form a conformal frame.
The correspondence between conformal harmonic discs in $D$ and holomorphic null discs
in $\Tcal_D$, mentioned above, shows that for all $\bx\in D$, $\by\in\R^n$ and
$(\bu,\bv)\in \CF_n$ we have that
\begin{equation}\label{eq:NM}
	\Ncal_D(\bx+\imath \by,\bu\pm \imath \bv) = \Mcal_D(\bx,(\bu,\bv)).
\end{equation}
This shows in particular that every extremal conformal harmonic
disc in $D$ is the real part of an extremal holomorphic null disc in the tube $\Tcal_D$.
Therefore, the correspondence between the extremal conformal minimal discs in 
the ball $\B^n\subset\R^n$ and the Kobayashi geodesics in the tube $\Tcal_{\B^n}$, 
used in the proof of Lemma \ref{lem:main}, extends to any bounded strongly convex domain 
$D\subset\R^n$ with $\Cscr^2$ boundary satisfying the following condition. 
The notion of a stationary holomorphic disc was explained in Remark \ref{rem:stationary}.

%
%
\begin{definition}\label{def:N} 
A domain $D\subset\R^n$ satisfies {\em Condition N} if for any 
point $\bx\in D$ and null vector $0\ne \bw \in\nullq^{n-1}$ there is a stationary 
holomorphic null disc in the tube $\Tcal_D$ through the point $\bx+\imath 0$ in the direction $\bw$.
\end{definition}

Our proof of Theorem \ref{th:main} implies the following. 

\begin{theorem} \label{th:N}
If $D$ is a bounded strongly convex domain in $\R^n$ $(n\ge 3)$ 
with smooth boundary satisfying Condition N, then for every point $\bx\in D$ and
$2$-plane $\Lambda\in  G_2(\R^n)$ there exists an extremal conformal harmonic
disc $f:\D\to D$ with $f(0)=\bx$ and $df_0(\R^2)=\Lambda$. Such $f$ is
unique up to a rotation of $\D$.
\end{theorem}

\begin{proof}
Let $0\ne \bw=\bu-\imath \bv\in \nullq^{n-1}$ be  such that $\Lambda=\span\{\bu,\bv\}$.
By Condition N there is a stationary holomorphic null disc $F:\D\to \Tcal_D$ with $F(0)=\bx+\imath 0$ 
and $F'(0)=\alpha \bw$ for $\alpha\in\C$, and $F$ is unique up to rotations of $\D$ by Lempert's
theorem \cite[Theorem 2]{Lempert1981}. The real part $f=\Re F:\D\to D$ is then a conformal
harmonic disc as in the theorem.
\end{proof}

\begin{problem} \label{prob:N} 
Which bounded strongly convex domains in $\R^n$, besides the ball, satisfy Condition N?
\end{problem}

Complex geodesics of the Kobayashi metric in tubes over convex domains $D\subset\R^n$ 
were studied by Zaj\k{a}c \cite{Zajac2015,Zajac2016}, Pflug and Zwonek 
\cite{PflugZwonek2018}, and Zwonek \cite{Zwonek2021X}. It would be of interest to see
whether these works can be used to give information on validity of Condition N.
The fact that Condition N holds on the ball $\B^n$ may simply be a lucky coincidence 
which makes our analysis work on this most symmetric domain.

%
%
\subsection*{Acknowledgements}
Forstneri\v c was partially supported by research program P1-0291 and grant J1-9104
from ARRS, Republic of Slovenia. Kalaj was partially supported by a research fund of
University of Montenegro. We wish to thank Antonio Alarc\'on for his remarks which led
to improved presentation, Dmitry Khavinson for information 
regarding harmonic self-maps of the disc, L\'aszl\'o Lempert and Stefan Nemirovski 
for information regarding projectively invariant metrics, 
Joaqu\'in P\'erez for consultation on the state of the art concerning growth of 
conformal minimal surfaces, and Sylwester Z{\k a}jac and 
W{\l}odzimierz Zwonek for communications
concerning geodesics of the Kobayashi metric in tube domains. The first named author
also thanks Miodrag Mateljevi\'c with whom he discussed this problem in 2016.



\vspace*{5mm}

\noindent Franc Forstneri\v c

\noindent University of Ljubljana, Faculty of Mathematics and Physics, 
Jadranska 19, SI--1000 Ljubljana, Slovenia, and
Institute of Mathematics, Physics and Mechanics, Jadranska 19, SI--1000 Ljubljana, Slovenia

\noindent e-mail: {\tt franc.forstneric@fmf.uni-lj.si}

\vspace*{5mm}
\noindent David Kalaj

\noindent University of Montenegro, Faculty of Natural Sciences and Mathematics, 81000, Podgorica, Montenegro

\noindent e-mail: {\tt davidk@ucg.ac.me}

\end{document}